\documentclass[10pt,reqno]{amsart} 

\usepackage{amssymb,latexsym}
\usepackage{cite} 

\usepackage[height=190mm,width=130mm]{geometry} 

\theoremstyle{plain}

\newtheorem{proposition}{Proposition}

\theoremstyle{definition}
\newtheorem{definition}{Definition}

\theoremstyle{remark}

\newcommand{\Z}{\mathbb{Z}}

\newcommand{\vac}{\mathbf {1}}
\newcommand{\delz}{\partial_{z}}

\numberwithin{equation}{section} 

\newcommand{\Tr}{\mathrm{Tr}}

\begin{document}

\title[Foliation of a space associated to vertex operator algebra]  
{Foliation of a space associated to vertex operator algebra}
\author{A. Zuevsky}
\address{Institute of Mathemtics \\ Czech Academy of Sciences\\ Zitna 25, Prague \\ Czech Republic}

\email{zuevsky@yahoo.com}







\begin{abstract}
We construct the foliation of aspace associated to correlation functions of vertex operator algebras on considered on 
Riemann surfaces.  
We prove that the computation of general genus $g$ correlation functions determines a foliation 
on the space associated to these correlation functions a sewn Riemann surface. 
Certain further applications of the definition are proposed. 
\end{abstract}

\vskip12pt  

\maketitle

\section{Introduction}
The theory of vertex operator algebras correlation functions considered on Riemann surfaces is a rapidly developing 
field of studies.
 Algebraic nature of methods applied in this field helps to understand and compute  
the structure of vertex operator algebras correlation functions.  
On the other hand, the geometric side of vertex operator algebra correlation functions 
 is in associating their formal parameters with 
local coordinates on a manifold. 
Depending on the geometry, one can obtain various consequences for a vertex operator algebra and its space 
of correlation functions.  
One is able to study the geometry of a manifold using 
 the algebraic structure of a vertex operator algebra defined on it, and, in particular, it is important to consider foliations of 
associated spaces.

In this paper we introduce the formula for an $n$-point function for a vertex operator algebra 
$V$ on a genus $g$ Riemann surface $\mathcal{S}^{(g)}$
 obtained as a result of 
sewing of lower genus surfaces $\mathcal{S}^{(g_1)}$ 
and $\mathcal{S}^{(g_2)}$ of genera $g_1$ and $g_2$, $g=g_1+g_2$. 
Using this formulation, we then introduce the construction of 
a foliation for the  space of correlation functions for vertex operator algebra 
with formal parameters defined on general sewn Riemann surfaces.  
Computations of a vertex operator algebra correlation functions allow us to define 
foliation of the space associated to a vertex operator algebra with formal parameter associated to a local 
coordinate on a genus $g$ Riemann surface sewn from lower genus Riemann surfaces. 
Such foliations are important both for studies of the space of correlation function 
for a vertex operator algebras and possibly 
for studies of smooth manifolds in the frames of Losik's approach \cite{LosikArxiv}. 
\section{Vertex operator algebras}
First, we recall the definitions of the standard formal series 
\begin{eqnarray}
\delta\left(\frac{x}{y}\right)&=&\sum_{n\in \mathbb{Z}} x^{n}y^{-n},
\label{delta}\\
(x+y)^{\kappa}&=&\sum_{m\ge 0}\binom{\kappa}{m}x^{\kappa-m}y^m,
\label{xykappa}
\end{eqnarray} 
for any formal variables $x,y,\kappa$ where $\binom{\kappa}{m}=\frac{\kappa(\kappa-1)\ldots (\kappa-m+1)}{m!}$.  
A vertex operator algebra \cite{K, FBZ}  
$(V,Y,\mathbf{1},\omega)$ consists of a $\Z$-graded complex vector space 
$V = \bigoplus_{n\in\Z}\,V_{(n)}$ where $\dim V_{(n)}<\infty$ for each $n\in \Z$, 
a linear map $Y:V\rightarrow {\rm End}(V)[[z,z^{-1}]]$ for a formal parameter $z$ and pair 
of distinguished vectors: the vacuum $\mathbf{1}\in V_{(0)}$ and the conformal vector $\omega\in V_{(2)}$. 
For each $v\in V$, the 
image under the map $Y$ is the vertex operator
\begin{align*}
Y(v,z) = \sum_{n\in\Z}v(n)z^{-n-1},
\end{align*}
with {modes} $v(n)\in {\rm End}(V)$, where $Y(v,z)\mathbf{1} = v+O(z)$.

The linear operators (modes) $u(n):V\rightarrow V$ satisfy creativity 
\begin{equation}
Y(u,z)\vac = u +O(z)
\label{create}
\end{equation}
and lower truncation 
\begin{equation}
u(n)v=0,
\label{lowertrun}
\end{equation}
for each $u$, $v\in V$ and $ n\gg 0$. 
Each vertex operator satisfies the translation property 
\begin{equation}
Y(L(-1)u,z)=\delz Y(u,z).  
\label{YL(-1)}
\end{equation}
Finally, the vertex operators satisfy the Jacobi identity
\begin{eqnarray}
\notag
&& z_0^{-1}\delta\left( \frac{z_1 - z_2}{z_0}\right) Y (u, z_1 )Y(v , z_2)    
  - z_0^{-1} \delta\left( \frac{z_2 - z_1}{-z_0}\right) Y(v, z_2) Y(u , z_1 ) 
\\
\label{VOAJac}
&& 
= z_2^{-1}    
\delta\left( \frac{z_1 - z_0}{z_2}\right)
Y \left( Y(u, z_0)v, z_2\right).   
\end{eqnarray} 
 Vertex operators satisfy locality, i.e., for all $u,v\in V$ there exists an integer $k\geq 0$ such that
\begin{align*}
(z_1 - z_2)^k \left[ Y(u,z_1), Y(v,z_2) \right] = 0.
\end{align*}
The vertex operator of the conformal vector $\omega$ is 
\begin{align*}
&Y(\omega,z) = \sum_{n\in\Z}L(n)z^{-n-2},
\end{align*}
where the modes $L(n)$ satisfy the Virasoro algebra with {central charge} $c$
\begin{align*}
&[L(m),L(n)] = (m-n)L(m+n) + \frac{m^3-m}{12}\delta_{m,-n}c\,\mathrm{Id}_V.
\label{eq:Ln}
\end{align*}
We define the homogeneous space of weight $k$ to be 
\begin{align*}
V_{(k)} = \{v\in V | L(0)v = kv\},
\end{align*}
and we write $
(v)=k$ for $v\in V_{(k)}$.
Amongst other properties, these axioms imply  
locality, associativity, commutativity and skew-symmetry:
\begin{eqnarray}
(z_{1}-z_{2})^m
Y(u,z_{1})Y(v,z_{2}) 
&=& (z_{1}-z_{2})^m
Y(v,z_{2})Y(u,z_{1}),
\notag
\\
&&
\label{Local}
\\
(z_{0}+z_{2})^n Y(u,z_{0}+z_{2})Y(v,z_{2})w &=& (z_{0}+z_{2})^n Y(Y(u,z_{0})v,z_{2})w,
\label{Assoc}
\\
u(k)Y(v,z)- Y(v,z)u(k)
&=& \sum_{j\ge 0}\binom{k}{j}
Y(u(j)v,z)z^{k-j},
\label{Comm}\\
Y(u,z)v &=&  e^{zL(-1)}Y(v,-z)u,
\label{skew}
\end{eqnarray}
for $u$, $v$, $w\in V$ and integers $m$, $n\gg 0$  \cite{K}.  

In  \cite{Zhu} Zhu introduced  a second "square-bracket" vertex operator algebra $(V,Y[,],\mathbf{1}
,\tilde{\omega})$ associated to a given vertex operator algebra $(V,Y(,),\mathbf{1},\omega )$.
  The new square bracket vertex operators are defined by a change of parameters, namely 
\begin{equation}
Y[v,z]=\sum_{n\in \mathbb{Z}}v[n]z^{-n-1}=Y(q_{z}^{L(0)}v,q_{z}-1),
\label{Ysquare}
\end{equation}
with $q_{z}=\exp (z)$, while the new conformal vector is 
$\tilde{\omega} =\omega -\frac{c}{24}\mathbf{1}$. 
For $v$ of $L(0)$ weight $wt(v)\in \mathbb{R}$ and $m\geq 0$, 
\begin{eqnarray}
v[m] &=&m!\sum\limits_{i\geq m}c(wt(v),i,m)v(i),  \label{square1} \\
\sum\limits_{m=0}^{i}c(wt(v),i,m)x^{m} &=&\binom{wt(v)-1+x}{i}.
\label{square2}
\end{eqnarray}
In particular we note that $v[0]=\sum\limits_{i\geq 0}\binom{wt(v)-1}{i}v(i)$. 
\subsection{The invariant form}
\label{liza}
The subalgebra $\{L(-1),L(0),L(1)\}\cong SL(2,\mathbb{C})$ associated with M\"{o}bius transformations on 
$z$ naturally acts on a vertex algebra \cite{K}. 
In particular, 
\begin{equation}
\gamma=\left(
\begin{array}{cc}
0 & 1\\
-1 & 0\\
\end{array}
\right), \quad   
 z\mapsto w= - \frac{1}{z},
 \label{eq: gam_lam} 
\end{equation}
is generated by 
\[
T = \exp \left( L(-1) \right) \exp \left( L(1) \right) \exp \left( L(-1) \right),
\]
  where  
\begin{equation}
T \; Y(u,z) \; T^{-1}= 
Y \left(  \left( \exp(-z) L(1) \right) 
\left(   -z \right)^{-2L(0)} u, -z^{-1}  \right).  \label{eq: Y_U} 
\end{equation}
Following \cite{Sche}, 
we therefore define 
\begin{equation}
Y^{\dagger }(u,z)=\sum_{n}u^{\dagger }(n)z^{-n-1}= T \; Y(u,z)\; T^{-1}. \label{eq: adj op} 
\end{equation}
For a quasi-primary vector $u$ (i.e., satisfying the condition $L(1) u=0$) of weight $wt(u)$, we have   
\begin{equation}
u^{\dagger }(n)=(-1)^{n+1}u(2wt(u)-n-2),  \label{eq: adj op qp}
\end{equation}
and 
 $L^{\dagger}(n)=(-1)^{n}L(-n)$.
\begin{definition}
We call a bilinear form $\langle .\ , .\rangle$ on $V$  
invariant if for all $a$, $b$, $u\in V$, \cite{Sche}  
\begin{equation}
\langle Y(u,z)a,b\rangle  = \langle a,Y^{\dagger }(u,z)b\rangle,  
\label{eq: inv bil form}
\end{equation}
i.e., 
\[
 \langle u(n)a, b\rangle= \langle a,u^{\dagger }(n)b\rangle.  
\]
Thus it follows that 
\[
\langle L(0)a,b\rangle =\langle a, L(0)b\rangle, 
\]
 so that 
$\langle a,b\rangle =0$  if $wt(a)\not=wt(b)$ for homogeneous $a$, $b$. One also finds 
$\langle a,b\rangle = \langle b,a \rangle$ \cite{Sche}. 
\end{definition}
The form  $\langle . \ , . \rangle$ is unique up to normalization if $L(1) V_{1}=V_{0}$.  
 We choose the normalization $\langle \mathbf{1}_V ,\mathbf{1}_V \rangle=1$. 
It is non-degenerate if and only if $V$ is simple \cite{L}. 
 Given any $V$ basis $\{ u^{\alpha}\}$
 we define the dual $V$ basis 
$\{ \overline{u}^{\beta}\}$ where $\langle u^{\alpha} ,\overline{u}^{\beta}\rangle_{\lambda}=\delta^{\alpha\beta}$. 

\section{Construction of an $n$-point function at a genus $g$ Riemann surface}
Recall that a conformal field theory defined on a Riemann surface  
$\mathcal{S}^{(g)}$ of genus $g$ \cite{knizka, TZ0, TZ1, TG, TW, MT}
 is determined 
by the set $\left\{Z^{(g)}_V\left(v_1, z_1; \ldots; v_n, z_n; \Omega^{(g)}\right)\right\}$ of all correlation functions 
 for all $n$, and all choices of points $z_i \in \mathcal{S}^{(g)}$, and all choices of elements $v_i$ of corresponding 
vertex operator algebra $V$.  

In this section, extending the genus two results of \cite{TZ0,  TZ1, TG}, we introduce the definition of  
 an $n$-point correlation function 
 for a vertex operator algebra $V$ on a genus $g$ Riemann surface $\mathcal{S}^{(g)}$
 obtained as a result of 
sewing of lower genus surfaces $\mathcal{S}^{(g_1)}$ and $\mathcal{S}^{(g_2)}$ of genera $g_1$ and $g_2$, $g=g_1+g_2$. 
The genus $g$ $n$-point correlation function for $a_1,\ldots,a_L \in V$ 
and $b_1,\ldots,b_R \in V$, $L+R=n$ inserted at $x_1,\ldots,x_L\in {\mathcal{S}}^{(g_1)}$ and 
$y_1,\ldots,y_R\in{\mathcal{S}}^{(g_1)}$, respectively, can be defined by
\begin{eqnarray}
\label{popa}
& Z^{(g)}_V\left(a_1,x_1;\ldots;a_L,x_L|b_1,y_1;\ldots;b_R,y_R;
\Omega_1^{(g_1)},\Omega_2^{(g_2)}, \epsilon, z_1, z_2\right) 
\notag
\\
 & = \sum_{n \ge 0} \sum_{u\in V} \epsilon^n Z^{(g_1)}_V\left(Y(a_1,x_1)\ldots Y(a_L,x_L)u;\Omega_1^{(g_1)}, z_1\right) 
\\
& \qquad \cdot Z^{(g_2)}_V\left(Y(b_R,y_R)\ldots Y(b_1,y_1)\overline{u};\Omega_2^{(g_2)}, z_2\right). 
\end{eqnarray}
Note that this construction of the correlation functions depends on the choice of insertion points 
$z_1 \in {\mathcal{S}}^{(g_1)}$ and $z_2\in {\mathcal{S}}^{(g_2)}$ of the $\epsilon$-sewing construction \cite{Y}.   
Here $\epsilon$ is the sewing complex parameter \cite{Y} and  $\Omega_i^{(g_i)}$, $i=1, 2$ 
are period matrices for Riemann surfaces
 ${\mathcal{S}}^{(g_i)}$.
To avoid misunderstanding, we say here, that the lower genus correlation functions 
$Z^{(g_1)}_V\left(Y(a_1,x_1)\ldots Y(a_L,x_L)u;\; \Omega_1^{(g_1)}, z_1\right)$ and  
 $Z^{(g_2)}_V\left(Y(b_R,y_R)\ldots\right.$  $\left. Y(b_1,y_1)\overline{u};\;  \Omega_2^{(g_2)}, z_2\right)$
are supposed to be know for given fixed $g_1$, $g_2$, or obtained via \eqref{popa} recursively, 
from known lower genus correlation functions (e.g., starting from the partition functions, or torus 
correlation functions). 
The explicit dependence on $\Omega^{(g_i)}$, $i=1$, $2$ is assumed. 
Not all vertex operator algebras admit  
 the notion of dual states. 
Thus we assume that there exists non-degenerate bilinear form on $V$ \cite{K}. 
On the right hand side the form is present in the definition of 
the dual state $\overline{u}$ with respect to such form.
 
%
There exists an algebraic procedure \cite{Zhu} relating 
 $n$-point correlation functions to a sum of $(n-1)$-point functions for vertex operator algebras 
on the torus. In \cite{Zhu} we find that 
the genus one $n$-point correlation functions obey
\begin{align*}
&Z^{(1)}_V(v_1, z_1; \ldots; v_n, z_n;\tau)
\notag
\\
&= \Tr_V \left( (v_1)_{wt(v_1) - 1} Y(q_{z_2}^{L(0)}v_2,q_{z_2})\ldots Y(q_{z_n}^{L(0)}v_n,q_{z_n}) q^{L(0)-c/24} \right)
\notag
 \\
&+ \sum_{k=2}^n \sum_{j\geq0} P_{1+j}(z_1-z_k,\tau) Z^{(1)}_V(v_2,z_2;\ldots; v_1[j]v_k,z_k;\ldots;v_n,z_n;\tau), 
\end{align*}
where $P_{l}(z,\tau)$ are Weierstrass functions \cite{L}, and square bracket on the right hand side denotes 
the deformed mode for vertex operator algebra on the torus \cite{Zhu}. 

A generalization \cite{TW} of the recursion procedure of \cite{Zhu} 
allows us to reduce a genus $g=g_1+g_2$
 $n$-point correlation function to the genus $g$ zero-point correlation function
 (the partition function) 
 on any higher genus Riemann surface 
 formed from two lower genus $g_1$, $g_2$ Riemann surfaces 
in the $\epsilon$ sewing procedure \cite{Y}.    
 Using explicit results of \cite{TZ, TW} for the Heisenberg vertex operator algebra in the Schottky formation of 
a genus $g$ Riemann surface, we conjecture the following general form of the 
reduction formula for a general vertex operator algebra 
considered on a genus $g$ Riemann surface 
obtained as a results of sewing of two lower genua $g_1$ and $g_2$ Riemann surfaces: 
\begin{eqnarray}
\label{gfu}
&& Z^{(g)}_V\left(v_1, z_1;\ldots; v_n, z_n; \mathcal T^{(g)} \right)   
\notag
\\
&& 
\qquad = \sum\limits_{k \ge 0}  \mathbb{P}_{k, n}^{(g)} \left(A^{(g_1)}_1, A^{(g_2)}_2\right) 
\cdot f_{V,k,n}(A)
\cdot \mathbb{O}_{k, n}^{(g)},   
\end{eqnarray}
where
 $\mathcal T^{(g)}=\left(\epsilon_1, \epsilon_2, z_1, z_2, \Omega^{(g_1)}, \Omega^{(g_1)} \right)$, 
are parameters of the genus $g$ $n$-point correlation function a sewn genus $g$ Riemann surface, 
$\mathbb P_{k,n}^{(g)}\left(A^{(g_1)}_1, A^{(g_2)}_2\right)$ is a polynomial of special infinite matrices 
$A^{(g_1)}_1$,  $A^{(g_2)}_2$, \cite{MT, TW}  
associated to the period matrices for Riemann surfaces ${\mathcal S}^{(g_1)}$ and 
${\mathcal S}^{(g_2)}$, 
$\mathbb{P}^{(g)}_{k, n}$, $\mathbb{O}^{(g)}_{k, n}$ are certain generalizations of classical elliptic functions
on higher genus Riemann surfaces depending on 
 arguments $(v_1, z_1; \ldots; v_n, z_n)$, and   
$f_{V,k,n}$ is a function of  the matrix $A=\left(I - A^{(g_1)}_1 \; A^{(g_2)}_2\right)$. 
Note that, using results of \cite{TZ, T}
 for the multiple-sewn sphere case, 
and as it is shown in \cite{TZ1}, the genus one case trace formulas can be obtain from the higher genus formulas 
(in particular, from genus two) 
 in the complex parameter $\rho$-sewing procedure \cite{Y}.  
\section{Construction of a foliation of the space of correlation functions 
over Riemann surfaces for a vertex operator algebra}
The construction of a vertex operator algebra $n$-point function of arbitrary genus gives us a 
hint how to define a special-type foliation 
of a space related to a vertex operator with the formal parameter 
associated to coordinates on a Riemann surface (formed from two lower genus surfaces) of genus $g$  
by means of vertex operator algebra correlation functions.  

 A $p$-dimensional 
foliation $\mathcal F$ of an $n$-dimensional 
 manifold $\mathcal M$ 
is a covering by a system of domains $\{U_i\}$ of $\mathcal M$ together with maps
\[
\phi_{i}: U_{i} \to \mathbb{C}^n, 
\]
such that for overlapping pairs $U_i$, $U_j$ the transition functions 
\[
\varphi_{ij} : {\mathbb C}^n \rightarrow {\mathbb C}^n, 
\]
defined by
\[
 \varphi_{ij}=\phi_{j} \phi_{i}^{-1}, 
\]
take the form 
\[
\varphi_{ij}(x,y)=\left(\varphi_{ij}^{1}(x), \varphi _{ij}^{2}(x,y)\right), 
\]
where $x$ denotes the first $n - p$ coordinates, and $y$ denotes the last $p$ co-ordinates. That is,
\[
{\begin{aligned}\varphi _{ij}^{1}:&\mathbb {C} ^{n-p}\to 
\mathbb {C} ^{n-p}, \\\varphi _{ij}^{2}:&\mathbb {C} ^{p}\to \mathbb {C}^{p}. \end{aligned}} 
\]

In this paper, we would like to foliate the space 
\[
\mathcal{M}_n= V^{\otimes^n} \times \mathcal{S}^{(g)},
\] 
where $V$ is a vertex operator algebra, and $\mathcal{S}^{(g)}$ is a Riemann surface of genus $g$. 
\begin{proposition}
The correlation functions \eqref{gfu} determine a foliation on the space 
$\mathcal{M}=\bigoplus_{n \ge 0} \mathcal{M}_n$. 
\end{proposition}
\begin{proof}
We consider a system of charts $\{U_m\}$, $m\in \mathbb N$, covering the Riemann surface part  
of the space $\mathcal{M}_n$ together with a filtration of the space $V^{{\otimes}^n}$ with respect 
to $z_i$, $i=1, \ldots, n$ belonging to $\{U_m\}$. This defined an infinite-dimensional analog of 
charts for $\mathcal M$. 
Suppose certain points $x_i$, $1 \le i \le p \le n$ are situated on $\mathcal{S}^{(g_1)}$-part and 
other $x_i$, $p+1 \le i \le n$ are on $\mathcal{S}^{(g_2)}$-part of the Riemann surface $\mathcal{S}^{(g)}$. 
On the Riemann surface $\mathcal{S}^{(g_1)}$ of genus $g_1$,  
let us choose a particular point $x_i$, $1 \le i \le p \le n$
 with associated local coordinate $z_i$ around $x_i$ which belongs to   
the domain $U_i$ of the system $\{U_m\}$. 
For another point $x_j$, $1 \le j\ne i \le p$ on $\mathcal{S}^{(g_1)}$-part we associate a domain  
$U_j$ intersecting with $U_i$. 
It can be always done since we can move points $x_i$, $1 \le i \le p$ around $\mathcal{S}^{(g_1)}$-part 
of the Riemann surface $\mathcal{S}^{(g)}$. 

Now let us compute a $p \le n$-point $V$-correlation functions $Z^{(g_1)}(v_1,z_1; \ldots ; v_p, z_p)$ by means of 
  the recursion procedure described above reducing $p$-point correlation functions to a one-point correlation function  
$Z^{(g_1)}(v_i, z_i)$ associated to our chosen point $x_i$ in the domain $U_i$.
Due to properties of vertex operators, the vertex algebra elements $(v_1, \ldots , v_p)$ can be chosen so that 
the expansion of the correlation function $Z^{(g_1)}(v_1,z_1; \ldots ; v_p,z_p)$ 
has a dimension $p$ polynomial nature, \cite{knizka}.  
The procedure of computation of $Z^{(g_1)}(v_1,z_1; \ldots ; v_p, z_p)$ by the 
 reduction to one-point correlation function  
$Z^{(g_1)}(v_i,z_i)$ defines a map
\[
\phi_i: V^{\otimes^p} \times  \mathcal{S}^{(g_1)} \rightarrow \mathbb{C}^p. 
\]
The result of computation of $Z^{(g_1)}(v_1,z_1; \ldots ; v_p,z_p)$ can 
 re-written to reduce to another one-point function  
$Z^{(g_1)}(v_j, z_j)$ associated to another coordinate $z_j$ around a point $x_j$ in the domain $U_j$.   
Thus, we can define the inverse map 
\[
\phi_j^{-1}: \mathbb{C}^p \rightarrow V^{\otimes^p} \times \mathcal{S}^{(g_1)}.   
\]
For any set $x$ of pairs of non-coinciding points among $x_1, \ldots, x_p$ on $\mathcal{S}^{(g_1)}$ we then define 
 the function for the intersecting domains $U_i$ and $U_j$ on the Riemann surface 
$\mathcal{S}^{(g_1)}$   
\[
 \varphi^{(g_1)}_{ij}(x)=\phi_{i} \phi_{j}^{-1}(x).  
\]
Exactly similar procedure is then performed for any set of $y$ non-coinciding points among $x_{p+1}, \ldots, x_n$  
 on the Riemann surface $\mathcal{S}^{(g_2)}$ of genus $g_2$ to define $\varphi^{(g_2)}_{ij}(y)$.  
Then the transition function is given by the map 
\[
\varphi_{ij}: \mathbb{C}^n \rightarrow \mathbb{C}^n, 
\]
\[
\varphi_{ij}(x,y)=\left(\varphi_{ij}^{(g_1)}(x), \varphi_{ij}^{(g_2)}(y)\right).  
\]
As a result, we have defined a foliation of the space $\mathcal{M}_n$ by means of correlation functions for corresponding 
 vertex operator algebras with the formal parameter associated to a coordinate on a Riemann surface. 
The total space $\mathcal{M}=\bigoplus_{n \ge 0} \mathcal{M}_n$ is foliated similar to $\mathcal{M}_n$. 
\end{proof}
The general situation is more complicated. Namely, we have to vary $n$, $p$, the set of vertex operator elements 
$\left\{ v_i\right\}$, $i=1, \ldots, n$, and $g_1$, $g_2$ with $g=g_1+g_2$  
(note also, that another parameter  
can be provided by the type of grading of the vertex operator algebra).

\section{Further directions}
The recursion procedure \cite{Zhu} 
 plays a fundamental role in the theory of correlation functions for 
 vertex operator algebras. It provides relations between $n$- and $n-1$-point correlation functions. 
In our foliation picture, the recursion procedure brings about relations among leaves of foliation. 
We work in the  
approach of formulation and computation of cohomologies of vertex operator algebras.   
Taking into account the above definitions and construction, 
we would like to develop a theory of 
characteristic classes for vertex operator algebras, and, in particular, for the space 
$\mathcal{M} = \bigoplus_{n=0}^{\infty} V^{\otimes^n} \times  \mathcal{S}^{(g)}$. 
This can have a relation with Losik's work \cite{LosikArxiv} 
proposing a new framework for singular spaces and 
new kind of characteristic classes.  
Losik defines a smooth structure on the leaf space
$M/\mathcal{F}$ of a foliation $\mathfrak{F}$ of codimension $n$ on
a smooth manifold $M$ that allows to apply to $M/\mathcal{F}$ the
same techniques as to smooth manifolds.
Losik 
defined the characteristic classes for a
foliation as  elements of the cohomology of certain bundles over the
leaf space $M/\mathcal{F}$. 
Similar to Losik's theory, we use certain bundles (of correlation functions) over a foliated space. 
Relations to \cite{BGG} on Reeb foliations modified Godbillon-Vey class and can be also considered. 

On the other hand, we would like to apply methodology of vertex algebras in order to 
complete Losik's theory of characteristic classes \cite{LosikArxiv}.  
One can mention a possibility to derive differential equations  
for correlation functions on 
separate leaves of foliation. Such equations are derived for various genera and can be used in frames 
of Vinogradov theory \cite{vinogradov}. 
The structure of foliation (in our sense) can be also studied from the automorphic function theory point of view. 
Since on separate leaves one proves automorphic properties of correlation functions, on can think about "global" automorphic 
properties for the whole foliation. 

Since we consider multipoint correlation function for vertex algebras on Riemann surfaces of arbitrary genus, 
there exists also a connection to 
Krichever-Novikov type algebras \cite{KN1, KN2, KN2, Schl0, Schl1, Schl2 }. These 
  are generalizations of the Witt, Virasoro, affine Lie algebras.  
In particular, one is able to use the structure of Krichever-Novikov type algebras (as higher-genus generalizations of 
algebras related to vertex algebras) to study foliated spaces of correlation functions introduced in this paper.
The construction of $n$-point functions on genus $g$ sewn Riemann surfaces 
can be used for introduction and computation of cohomology of Krichever-Novikov type algebras in appropriate setup. 
We plan to fully shed light on this subject in a future paper. 

\section*{Acknowledgments}
The author would also like to thank a refereee provided comments to the first version of the manuscript. 


\end{document}